\documentclass{article}

\usepackage[12pt]{extsizes}
\usepackage{cmap}
\usepackage[utf8]{inputenc}
\usepackage[T2A]{fontenc}
\usepackage[english]{babel}
\usepackage{amsmath}
\usepackage{amsthm}
\usepackage{amssymb}
\usepackage{amsfonts}
\usepackage{euscript}
\usepackage{enumerate}
\usepackage{hyperref}
\usepackage{tikz-cd}

\usepackage[all]{xy}
\usepackage{epigraph}

\newcommand{\R}{\mathbb{R}}
\renewcommand{\C}{\mathbb{C}}

\newcommand{\Z}{\mathbb{Z}}

\newcommand{\gr}{\mathop{\mathrm{gr}}\nolimits}

\newcommand{\SL}{\mathop{\mathrm{SL}}\nolimits}

\newcommand{\GL}{\mathop{\mathrm{GL}}\nolimits}

\newcommand{\Spec}{\operatorname{Spec}}

\newcommand{\mf}{\mathfrak}

\newcommand{\mc}[1]{\mathcal{#1}}

\renewcommand{\Re}{\operatorname{Re}}

\newcommand{\End}{\operatorname{End}}

\newcommand{\Res}{\operatorname{Res}}

\newcommand{\abs}[1]{\lvert #1\rvert}

\newcommand{\boldt}{\bold t}

\newcommand{\rmi}{\mathrm{i}}

\newcommand{\Supp}{\operatorname{Supp}}
\newcommand{\Gr}{\operatorname{Gr}}

\DeclareMathOperator{\sph}{sph}
\DeclareMathOperator{\T}{T}
\begin{document}
\newtheorem{thr}{Theorem}[section]
\newtheorem*{thr*}{Theorem}
\newtheorem{lem}[thr]{Lemma}
\newtheorem*{lem*}{Lemma}
\newtheorem{cor}[thr]{Corollary}
\newtheorem{prop}[thr]{Proposition}
\newtheorem*{prop*}{Proposition}
\newtheorem{stat}[thr]{Statement}
\newtheorem*{stat*}{Statement}
\newtheorem{example}[thr]{Example}

\theoremstyle{definition}
\newtheorem{defn}[thr]{Definition}
\newtheorem*{defn*}{Definition}
\theoremstyle{remark}
\newtheorem{rem}[thr]{Remark}
\newtheorem*{rem*}{Remark}
\numberwithin{equation}{section}
\title{Residue construction of quantized Coulomb branches}
\author{Daniil Klyuev}

\maketitle
\begin{abstract}
We describe the image of a quantized BFN Coulomb branch $\mc{A}_{G,N}^{\hbar=1}$ under localization (abelianization) map for any $G,N$. In most cases, such as quiver gauge theories without loops, this description works for any flavors, sometimes we need to take generic or formal flavor parameters. The answer is given by a roots and residue condition similar to Ginzburg---Kapranov---Vasserot construction of DAHA~\cite{GKV}. As a corollary, for any quantized conical Coulomb branch (and zero or small flavors) we provide a mathematical construction of the sphere trace introduced by Gaiotto and Okazaki~\cite{GO}.
\end{abstract}
\section{Introduction}

Let $G$ be a reductive group and $N$ its representation. Braverman, Finkelberg and Nakajima gave a mathematical definition of Coulomb branch $\mc{M}_{G,N}=\Spec \mc{A}_{G,N}^{\hbar=0}$ in~\cite{BFN}. The algebra $\mc{A}_{G,N}^{\hbar=0}$ is defined as $G(O)$-equivariant Borel-Moore homology of an infinite-dimensional scheme. Adding equivariance with respect to the loop torus $\C^{\times}_{\hbar}$ we get an algebra $\mc{A}_{G,N}$ that is a deformation of $\mc{A}_{G,N}^{\hbar=0}$ over $\C[\hbar]$. Below we will set $\hbar=1$. 

In the case of quiver gauge theories for quivers without loops quantized Coulomb branches are isomorphic to a quotient of a shifted Yangian~\cite{BFNQuivers, KTWWY} (hence they are called truncated shifted Yangians.) The main tool in the proof of this description is localization map $(i_*)^{-1}\colon \mc{A}_{G,N}\hookrightarrow \mc{A}_{T,N,loc}$ described in Remark~5.23 of~\cite{BFN}.  Here "$loc$" means that we consider Ore localization at $\alpha+k\in\C[\mf{t}]$ for roots $\alpha$ of $G$ and integers $k$.

We can further embed $\mc{A}_{T,N,loc}$ into $\mc{A}_{T,0,loc}=\oplus_{\lambda\in Y} \C[\mf{t}]_{loc}r^{\lambda}$. Here $Y$ is the cocharacter lattice of $T$. In this paper we describe the image of $\mc{A}_{G,N}$ in $\mc{A}_{T,0,loc}$ for any pair $(G,N)$ obtaining a roots and residues description similar to Ginzburg---Kapranov---Vasserot residue construction of DAHA~\cite{GKV}.

 An answer for a very close object in the case of $G=\SL(2)$ and $N=0$ was given in~\cite{KV}. Schrader and Shapiro proved the residue description for pure $K$-theoretic $\GL_n$ Coulomb branch, see Lemma~6.14 in~\cite{SS}. Crisan obtained roots and residue description of Iwahori Coulomb branches in~\cite{Dragos}.

 Let $\mc{A}_{G,N}$ be quantized Coulomb branch with generic or formal flavor parameters.
 \begin{rem}
 This is a technical condition, the only thing we actually need is that there are no cancellations in~\eqref{EqFractionDoubled}. This is true for any flavors if no weight of $N$ is proportional to a root of $G$ and this is true for zero flavors unless $2\xi_i$ is a root for some weight $\xi_i$ of $N$.

\end{rem}  Our main result is the following:
\begin{thr*}[Theorem~\ref{ThrAnswer}]
 Let $S$ be a subset of $W$-invariant elements $a=\sum a_{\lambda} r_{\lambda}$ of $\mc{A}_{T,0,loc}$ that satisfy the following condition: 
\begin{enumerate}
\item
Each $a_{\lambda}$ is divisible by 
\begin{equation}
\label{EqFractionDoubled}
\frac{\prod_{i=1}^n \prod_{j=0}^{-\xi_i(\lambda)-1}(\xi_i+b_i+(\xi_i(\lambda)+j+\tfrac12))}{\prod_{\alpha>0}\prod_{\abs{k}<N}(\alpha-k)}
\end{equation} 
for some $N>0$. In particular, $a_{\lambda}$ has at most simple poles at $\alpha-k$, where $\alpha$ is a positive root of $G$ and $k$ is an integer. 

Denote the residue at this divisor by $\Res_{\alpha,k}(a_{\lambda})\in \C[\alpha^{\perp}]$.

\item
The residues satisfy the condition $\Res_{\alpha,k}(a_{\lambda})+\Res_{\alpha,k}(a_{s_{\alpha}(\lambda)+k\check{\alpha}})=0$.
\end{enumerate}

Then the image of $\mc{A}_{G,N}$ in $\mc{A}_{T,0,loc}$ is $S$.
\end{thr*}

One could use the residue construction to better understand twisted and positive traces on quantized Coulomb branches, see introduction to~\cite{KAb} or~\cite{KV} for a discussion of twisted and positive traces. The $\SL(2)$ residue construction was a crucial ingredient in the second part of~\cite{KV}. The only thing we discuss here is the following corollary of the residue construction.

Suppose that $G,N$ define a {\it conical} Coulomb branch (the theory is {\it good} or {\it ugly} in physics terminology. We discuss these terms in Section~\ref{SecSphere}.) We provide a mathematical proof that the special sphere partition function introduced by Gaiotto and Okazaki (see Section~2.6 and formula~(2.9) in~\cite{GO}) can be extended to the  {\it sphere trace} map $\T_{\sph}\colon \mc{A}_{G,N}\to \C.$ We also show that $\T_{\sph}$ indeed satisfies the twisted trace condition:
\begin{thr*}(Theorem~\ref{ThrSphereTraceCohomology})
Suppose $\mc{M}_{G,N}$ is conical and $\abs{\Re b_i}<\tfrac12$ for $i=1,\ldots,n$. Then there exists a $g_{\text{parity}}$-twisted trace $T_{\sph}\colon \mc{A}_{G,N}\to \C$ such that for $R(x)\in H^*_G(pt)\subset \mc{A}_{G,N}$ we have
\[T(R(x))=\int_{\mf{t}_{\R}}R(x)w_0(x)dx,\] where \[w_0(x)=\frac{\prod_{\alpha>0} \sinh(\pi\alpha)^2}{\prod_{j=1}^n \cosh(\pi(\xi_j-\rmi b_j))}.\]
\end{thr*}
\begin{rem}
The residue construction, the sphere trace and the proof in the article work with minimal changes for $K$-theoretic Coulomb branches. There is no conicity condition in that case.

The results can also be modified to work for algebras defined in appendix to~\cite{NW}. In particular, it works for quiver Coulomb branches with symmetrizers in the case when the definition in the main part of~\cite{NW} gives the algebra isomorphic to the algebra defined in appendix (which is true in all finite types, see Theorem~C.7 in~\cite{NW}.)

The residue construction is modified as follows. The roots of $\hat{\mf{g}_k}$ should live in $\mf{t}_k^*[[z^k]]$. So we allow at most simple poles at $\alpha-kl$ for $\alpha$ a root of $\mf{g}_k$ and integer $l$. The Euler classes appearing in the numerator of~\eqref{EqFractionDoubled} should come from quotients of $\oplus \mathbf{N}_k[[z^k]]z^{\tfrac k2}$.

In order to compute the sphere trace when $b_i=0$ we should take \[w_0(x)=\frac{\prod_{k=1}^d\prod_{\alpha\in\Delta_{k,+}}\sinh(\frac{\pi\alpha}{k})^2}{\prod_{k=1}^d\prod_{j=1}^{\dim \mathbf{N}_k}\cosh(\frac{\pi\xi_{kj}}{k})}.\]

\end{rem}
\subsection*{Acknowledgments}
I am grateful to Gus Schrader for suggesting the statement of Proposition~\ref{PropResidueConditionIsPreservingPolynomials}, other useful discussions and for the lecture course on $K$-theoretic Coulomb branches in Spring quarter of 2025 at Northwestern University. I am grateful to Vasily Krylov and to Dragos Crisan for comments on the previous versions of the paper.

\section{Description of pure Coulomb branch}
\label{SecPureResidues}
Using homological version of Lemma~\ref{LemLocalizationIsCompatibleWithAction} we have an inclusion \[\mc{A}_{G,0}\to \mc{A}_{T,0,loc}\cap \End(\C[\mf{t}]^W).\] The action of $\mc{A}_{T,0}$ on $\C[\mf{t}]$ is given by homological version of Proposition~\ref{PropActionTZero}. Namely, the element $r^{\lambda}\in \mc{A}_{T,0}$ sends $P(\boldt)\in \C[\mf{t}]$ to $P(\boldt+\lambda).$ Let us describe $\mc{A}_{T,0,loc}\cap \End(\C[\mf{t}]^W)$ explicitly.
\begin{prop}
\label{PropResidueConditionIsPreservingPolynomials}
Let $S$ be a subset of $W$-invariant elements $a=\sum a_{\lambda} r_{\lambda}$ of $\mc{A}_{T,0,loc}$ that satisfy the following conditions: 
\begin{enumerate}
\item
Each $a_{\lambda}$ has at most simple poles at $\alpha-k$, where $\alpha$ is a root of $G$ and $k$ is an integer. Denote the residue at this divisor by $\Res_{\alpha,k}(a_{\lambda})\in \C[\alpha^{\perp}]$.
\item
The corresponding residues satisfy the condition \[\Res_{\alpha,k}(a_{\lambda})+\Res_{\alpha,k}(a_{s_{\alpha}(\lambda)+k\check{\alpha}})=0.\]
\end{enumerate}
Then $S$ coincides with the subalgebra of $\mc{A}_{T,0,loc}$ that preserves $\C[\mf{t}]^W$.
\end{prop}
\begin{proof}
First, we check that $S$ indeed sends each $W$-invariant polynomial to a polynomial. Let $a$ be an element of $S$ and $P$ be an invariant polynomial. We have to check that $aP=\sum a_{\lambda}P(\boldt+\lambda)$ is regular. Since each $a_{\lambda}$ has at most simple poles at $\alpha-k$, we have to check that this sum has no pole at $\alpha-k$. We have \begin{align}P(\boldt+\lambda)|_{\alpha(\boldt)=k}=(s_{\alpha}P)(\boldt+\lambda)|_{\alpha(\boldt)=k}=\\
P(s_{\alpha}\boldt+s_{\alpha}\lambda)|_{\alpha(\boldt)=k}=P(\boldt-k\check{\alpha}+s_{\alpha}(\lambda))|_{\alpha(\boldt)=k}=\\
r_{s_{\alpha}(\lambda)-k\check{\alpha}}(P)|_{\alpha(\boldt)=k}.
\end{align}
Hence in the computation of the pole we get the term $a_{\lambda}+a_{s_{\alpha}(\lambda)-k\check{\alpha}}$, it is regular.

Now we argue in the opposite direction. Let $a=\sum a_{\lambda}r_{\lambda}\in\mc{A}_{ab,loc}$ be an element that preserves the space of $W$-invariant polynomials. By construction, the denominator of each $a_{\lambda}$ is a product of linear factors of the form $\alpha-k$. We have to prove that there are no poles of order at least two and that no expression of the form $a_{\lambda}+a_{s_{\alpha}(\lambda)-k\check{\alpha}}$ has a pole at $\alpha-k$.

Suppose that at least one of these conditions is not satisfied for $a_{\lambda_0}$ and $\alpha-k$. Let $\lambda_1=s_{\alpha}(\lambda_0)-k\check{\alpha}$. Let $\Supp(a)=\{\lambda\mid a_{\lambda}\neq 0\}\cup \{\lambda_1\}$.

For any natural numbers $N,M$, distinct points $\boldt_1,\boldt_2,\ldots,\boldt_M$ and polynomials $Q_1,\ldots,Q_M$ there exists a polynomial $P$ such that $P-Q_i$ belongs to $N$-th power of the ideal corresponding to $\boldt_i$.

Let $\boldt$ be a point satisfying $\alpha(\boldt)=k$. Let $S$ be any set of representatives for $W/\langle s_{\alpha} \rangle$ in $W$. For Zariski generic $\boldt$ the points $w(\boldt+\lambda), w\in S, \lambda\in \Supp(a)$, are distinct and the only elements in $W(\boldt+\lambda_0)$ of the form $\boldt+\lambda$ with $\lambda\in \Supp(a)$ are $\boldt+\lambda_0$ and $\boldt+\lambda_1$. We also require $\beta(\boldt+\lambda_0)+l, \beta(\boldt+\lambda_1)+l\neq 0$ for all possible linear factors $\beta+l$ with $\beta\neq\pm\alpha$ appearing in the denominators of $a_{\lambda}$, the set of $\boldt$ satisfying all three conditions is still Zariski generic. %Indeed, the elements $w\in W/\langle s_{\alpha}\rangle$ give $\frac{\abs{W}}{2}$ distinct maps from $\alpha^{\perp}$ to $\mf{t}$, so if we take $\boldt_1$ far enough from a finite number of hyperplanes in $\alpha^{\perp}$, the distance between $w\boldt_1$ for distinct $w\in W/\langle s_{\alpha}\rangle$ will be larger than the maximum of $\abs{\lambda}$ with $a_{\lambda}\neq 0$.

We choose $N$ equal to the largest possible degree of a denominator for all nonzero $a_{\lambda}$. We take $M$ equal to the number of distinct points of the form $w(\boldt+\lambda)$, $w\in W$, $\lambda\in \Supp(a)$. Suppose that $\boldt+\lambda_0=\boldt_1$. We set $Q_i=0$ for $i>1$.

It follows that $\sum wP$ is a $W$-invariant polynomial equal to $1$ on $W\boldt_1$ and equal to $0$ on other points of the form $w(\boldt+\lambda)$, up to $N$-th order. In particular, for $\lambda\in \Supp(a)$, the value $P(\boldt+\lambda)$ zero up to $N$-th order unless $\lambda=\lambda_0$ or $\lambda=\lambda_1$.

We have $(r_{\lambda}P)(\boldt_1)=P(\boldt_1+\lambda)$. With our choice of $N$, the functions $(a_{\lambda}r_{\lambda}P)$ are regular at $\boldt=\boldt_1$ for $\lambda\neq \lambda_0,\lambda_1$. For $\lambda_0$ and $\lambda_1$ we write $a_{\lambda_i}=(\alpha-k)^{-L_i}b_{i}$ with $L_i>0$ and $b_i$ regular at $\boldt_1$. The singular part of $(aP)$ at $\boldt_1$ can be computed as $(\alpha-k)^{-L_0}b_0Q_0+(\alpha-k)^{-L_1}b_1Q_0(s_{\alpha}\boldt+k\check{\alpha})$. It should be zero for all polynomials $Q_0$. Taking $Q_0=1$ gives $L_0=L_1$, $b_0(\boldt_1)+b_1(\boldt_1)=0$. Suppose that $L_0>1$. Take $Q_0=\alpha-k$. We get $Q_0(s_{\alpha}\boldt+k\check{\alpha})=k-\alpha$, so that \[(\alpha-k)^{-L_0}b_0Q_0+(\alpha-k)^{-L_1}b_1Q_0(s_{\alpha}\boldt+k\check{\alpha})=(\alpha-k)^{1-L_0}(b_0-b_1).\] This expression is not regular at $\boldt_1$.

We get that $L_0=L_1=1$ and $b_0(\boldt_1)+b_1(\boldt_1)=0$. Since the space of all possible $\boldt_1$ is Zariski dense in $\alpha^{\perp}$, we have \[\Res_{\alpha,k}a_{\lambda_0}+\Res{\alpha_k}a_{\lambda_1}=(b_0+b_1)|_{\alpha^{\perp}}=0.\]

%Indeed, suppose $w\neq 1,s_{\alpha}$. Then $w$ acts nontrivially on the plane $\alpha(\boldt)=0$. 
\end{proof}
\begin{prop}
\label{PropImageOfAZeroResidue}
The image of $\mc{A}$ in $\mc{A}_{T,0,loc}$ is $S$.
\end{prop}
\begin{proof}
The Coulomb branch algebra $\mc{A}$ has a filtration by coweights. The algebra $S$ also has a filtration by coweights: an element $a=\sum a_{\lambda}r_{\lambda}$ belongs to $S_{\leq \mu}$ if $\Supp(a)$ is contained in the set of weights of representation of highest weight $\mu$. By construction, the inclusion $\mc{A}\subset S$ respects the filtration.

The graded piece $(\gr\mc{A})_{\lambda}$ is isomorphic to $H_{BM}^{G(O)}(\Gr_{\lambda})=\C[\mf{t}]^{W_{\lambda}}$. Take any lift $a\in\mc{A}_{\leq \lambda}$ of $1\in\C[\mf{t}]^{W_{\lambda}}$. Consider the image of $a$ in $\mc{A}_{T,0,loc}$, write $a=\sum a_{\mu}r_{\mu}$. Then $a_{\lambda}$ does not depend on the choice of $a$ and can be computed using pullback with respect to inclusion $t^{\lambda}\hookrightarrow \Gr_{\lambda}$ of a smooth point. The result will be $a_{\lambda}=\frac{1}{\prod_{\alpha}\prod_{j=0}^{\alpha(\lambda)-1}(\alpha-j)}$

Take any $b\in S_{\leq \lambda}$. Suppose that $b_{\lambda}$ has a pole at $\alpha-k$. Then $b_{s_{\alpha}(\lambda)+k\check{\alpha}}$ is nonzero. Hence $s_{\alpha}(\lambda)+k\check{\alpha}$ is a weight of the representation $V_{\lambda}$ of Langlands dual group $\check{G}$. Applying $s_{\alpha}$ we get $\lambda-k\check{\alpha}$. The weights of $V_{\lambda}$ of the form $\lambda-l\check{\alpha}$ are $\lambda,\lambda-\alpha,\cdots, s_{\alpha}(\lambda)=\lambda-\alpha(\lambda)\check{\alpha}$. It follows that $0\leq k\leq \alpha(\lambda)$. Also, when $s_{\alpha}(\lambda)+k\check{\alpha}=\lambda$, the corresponding residue equals to zero, hence $0\leq k<\alpha(\lambda)$. We deduce that $b_{\lambda}$ is divisible by $a_{\lambda}$. Since both $a,b$ are $W$-invariant, the ratio $\frac{b_{\lambda}}{a_{\lambda}}$ is a $W_{\lambda}$-invariant polynomial. It follows that $\gr\mc{A}$ surjects onto $\gr S$. Therefore the image of $\mc{A}$ is the whole $S$, as claimed.
\end{proof}
\section{General case}

We turn to an arbitrary Coulomb branch of cotangent type. Let $N$ be a representation of $G$. We decompose $N$ as a sum of irreducible $G$-representations and then further decompose each of them as a sum of weight spaces of $T$. Denote the $T$-weights of $N$ by $\xi_1,\ldots,\xi_n$. Choose a flavor for each irreducible summand of $N$ and define $b_i$ to be the flavor corresponding to the representations containing $\C\xi_i$. We require that $b_i$ are variables or generic numbers.

 Consider the following commutative diagram:
\[
\begin{tikzcd}
\mc{R}_{T,N} \ar[r,"\iota_T"]\ar[d,"i"] & \mc{T}_{T,N}\ar[d,"i_0"]\\
\mc{R}_{G,N} \ar[r,"\iota"] & \mc{T}_{G,N}
\end{tikzcd}
\]
Here $\mc{T}_{G,N}$ is the infinite-dimensional vector bundle over $\Gr_G$ constructed in~\cite{BFN} and $\iota$ is the inclusion $\mc{R}_{G,N}\subset \mc{R}_{T,N}$. The map $i_0$ lifts the inclusion $\Gr_T\subset\Gr_G$. This diagram gives rise to the following diagram of algebras, where all maps are injective

\begin{equation}
\label{EqDiagramOfALgebras}
\begin{tikzcd}
\mc{A}_{T,N,loc} \ar[r,"\iota_{T*,loc}"] & \mc{A}_{T,0,loc}\\
\mc{A}_{G,N}\ar[u,"i_*^{-1}"]\ar[r,"\iota_*"] & \mc{A}_{G,0}\ar[u,"i_{0*}^{-1}"]
\end{tikzcd}
\end{equation}

 Let $S$ be a subset of $W$-invariant elements $a=\sum a_{\lambda} r_{\lambda}$ of $\mc{A}_{T,0,loc}$ that satisfy the following condition: 
\begin{enumerate}
\item
Each $a_{\lambda}$ is divisible by 
\begin{equation}
\label{EqFraction}
\frac{\prod_{i=1}^n \prod_{j=0}^{-\xi_i(\lambda)-1}(\xi_i+b_i+(\xi_i(\lambda)+j+\tfrac12))}{\prod_{\alpha>0}\prod_{\abs{k}<N}(\alpha-k)}
\end{equation} 
for some $N>0$. In particular, $a_{\lambda}$ has at most simple poles at $\alpha-k$, where $\alpha$ is a positive root of $G$ and $k$ is an integer. 

Denote the residue at this divisor by $\Res_{\alpha,k}(a_{\lambda})\in \C[\alpha^{\perp}]$.
\item
The residues satisfy the condition $\Res_{\alpha,k}(a_{\lambda})+\Res_{\alpha,k}(a_{s_{\alpha}(\lambda)+k\check{\alpha}})=0$.
\end{enumerate}
\begin{thr}
\label{ThrAnswer}
The image of $\mc{A}_{G,N}$ in $\mc{A}_{T,0,loc}$ is $S$.
\end{thr}
\begin{proof}
Since $b_i$ generic or variables there are no cancellations between the numerator and the denominator in~\eqref{EqFraction}.

Using~\eqref{EqDiagramOfALgebras} we get that the image of $\mc{A}_{G,N}$ lies in the intersection of the images of $\mc{A}_{G,0}$ and $\mc{A}_{T,N,loc}$ in $\mc{A}_{T,0,loc}$. The map $\iota_{T*}$ is computed in formula (4.10) in~\cite{BFN}, and the image of $\iota_{T*}$ consists of elements $\sum a_{\lambda} r^{\lambda}$ such that $a_{\lambda}$ is divisible by $\prod_{i=1}^n \prod_{j=0}^{-\xi_i(\lambda)-1}(\xi_i+b_i+(\xi_i(\lambda)+j+\tfrac12))$. The statement about simple poles and the condition on the residues follows from the Proposition~\ref{PropImageOfAZeroResidue}. Hence $\mc{A}_{G,N}$ lands inside $S$.

It remains to show that the image of $\mc{A}_{G,N}$ is the whole $S$. Similarly to the above, we can compute the leading term of the localization of a lift $a\in\mc{A}_{\leq \lambda}$ of $1\in\C[\mf{t}]^{W_{\lambda}}\cong \mc{A}_{\leq \lambda}/\mc{A}_{<\lambda}$. The result is $\frac{r_{\lambda}}{\prod_{\alpha}\prod_{j=0}^{\alpha(\lambda)}(\alpha-j)}\in \mc{A}_{T,N,loc}$. Then we apply $\sigma_T^*$. The formula (4.10) in~\cite{BFN} says that \[\sigma_T^*(r^{\lambda})=\prod_{i=1}^n \prod_{j=0}^{-\xi_i(\lambda)-1}(\xi_i+b_i+(\xi_i(\lambda)+j+\tfrac12)).\] It follows that $\mc{A}_{G,N}$ surjects onto $S$. 
\end{proof}
\section{Sphere trace}
\label{SecSphere}
Let $\mc{A}$ be a noncommutative algebra over $\C$ and $g$ be an automorphism of $\mc{A}$. A linear map $T\colon \mc{A}\to\C$ is called {\it $g$-twisted trace} if $T(ab)=T(bg(a))$ for all $a,b\in \mc{A}$. Homological Coulomb branches are $\Z/2$-graded by the parity of homology group and below we take $g=g_{\text{parity}}$, the parity automorphism of $\mc{A}_{G,N}.$

We assume that the Coulomb branch $\mc{M}_{G,N}$ is {\it conical}. This means that $\mc{A}_{G,N}^{\hbar=0}$ is a positively graded algebra. For analytic reasons, we will take $b_i$ satisfying $\abs{\Re b_i}<\tfrac 12$. 

 \begin{defn}
     Suppose that $\mc{A}_{G,N}$ satisfies the conclusion of Theorem~\ref{ThrAnswer}. We define a linear map $\operatorname{T}_{\sph}\colon\mc{A}\to\C$ by
      \[\operatorname{T}_{\sph}(a)=\int_{\mf{t}_{\R}} a_0(\rmi x)w_0(x)dx,\] where 
      \[w_0(x)=\frac{\prod_{\alpha>0} \sinh(\pi\alpha)^2}{\prod_{j=1}^n \cosh(\pi(\xi_j-\rmi b_j))}\]
 \end{defn}
\begin{prop}
\label{PropTSph}
    The map $T_{\sph}$ is well-defined and one can extend the definition of $T_{\sph}$ to any flavors $b_1,\ldots,b_n$ satisfying $\abs{b_j}<\tfrac12$ for all $j$. Then the map $T_{\sph}$ is a $g_{\text{parity}}$-twisted trace.
\end{prop}
\begin{proof}
We start with the case when $b_j$ are generic numbers satisfying the condition above.

We will use the alternative definition of being conical: for all $\beta\in\mf{t}_{\R}$ we have $\sum_{i=1}^n\abs{\xi_i(\beta)}>2\sum_{\alpha>0}\abs{\alpha(\beta)}$. The equivalence between this definition and the standard one is proved in Lemma~2.6 of~\cite{GWTraces}.

We show that $w_0(x)$ exponentially decays at infinity. Let $tx\in\mf{t}_{\R}$ be a point with $\abs{x}=1$ and $t>0$. Then the numerator of $w_0$ is equivalent to $e^{t\sum_{\alpha>0}\abs{\pi\alpha(x)}}$ and the denominator is equivalent to $e^{t\sum_{i=1}^n \abs{\xi_i(x)}}$, hence $w_0(x)$ is equivalent to $e^{-ts}$, where $s$ is a positive number. Moreover, since $\sum_{\alpha>0}\abs{\alpha(x)}-\sum_{i=1}^n \abs{\xi_i(x)}$ is a continuous function, we have $s>s_0>0$ for some $s_0$ that does not depend on $x$. We see that $w_0(x)$ exponentially decays at infinity.

    The only possible poles of $a_{\lambda}$ have form $\alpha-k$ for a positive root $\alpha$ and an integer $k$. It follows from Theorem~\ref{ThrAnswer} that $a_0$ does not have a pole at $\alpha$. Hence $a_0$ is regular on the locus $\alpha(x)\in \rmi\R.$ Combining this and the exponential decay of $w_0$ we get that $\T_{\sph}(a)$ is well-defined.
    
    Another corollary of Theorem~\ref{ThrAnswer} is that the poles of $a_{\lambda}$ are simple. Then the trace condition $\operatorname{T}_{\mathrm{sph}}(ab)=\operatorname{T}_{\mathrm{sph}}(bg(a))$ can be checked for each term \[\int P_{\lambda}(\rmi x+\rmi\frac{\lambda}{2})a_{\lambda}(\rmi x)b_{-\lambda}(\rmi x+\rmi\lambda)w_0(x) dx= e^{2\pi\rmi\zeta(\lambda)}\int P_{\lambda}(\rmi x-\rmi\frac{\lambda}{2})a_{\lambda}(\rmi x-\rmi\lambda)b_{-\lambda}(\rmi x)w_0(x) dx\] similarly to Proposition~2.6 in~\cite{KAb}. Here $P_{\lambda}(x+\frac{\lambda}{2})=r^{\lambda}r^{-\lambda}$ comes from the multiplication in the abelian Coulomb branch. The poles of $a_{\lambda}$ and $b_{\lambda}$ are canceled by the zeroes of $w_0(x)$. The equation $w_0(x+\rmi\lambda)=e^{2\pi\rmi\zeta(\lambda)}w_0(x)$ holds for $\zeta=\tfrac12\sum_{i=1}^n\xi_i$ because $\cosh(x+\pi\rmi)=-\cosh(x).$
    
  It remains to show the statement of the proposition without genericity assumption. Let $S'$ be the subset of $\mc{A}_{T,0,loc}$ consisting of elements $\sum a_{\lambda}r^{\lambda}$ such that each $a_{\lambda}$ satisfies the divisibility condition above (this means $a\in\mc{A}_{T,N,loc}\subset\mc{A}_{T,0,loc}$) and has simple poles; without residue condition. The set $S'$ can be defined in the case when the flavors $b_1,\ldots,b_n$ are independent variables.

 Since different $a_{\lambda}$ can be chosen independently, we can write a spanning set $v_{i,\lambda}=f_{i,\lambda}r^{\lambda}$ of $S'$ over $\C[b_1,\ldots,b_n]$. Then $f_{i,0}(\rmi x)w_0(x)$ is a function in $b_1,\ldots,b_n$ holomorphic on the set $\abs{\Re b_j}<\tfrac12, j=1,\ldots,n$. It follows that \[\int_{\mf{t}_{\R}} f_{i,0}(\rmi x)w_0(x)dx.\] is also holomorphic on the set $\abs{\Re b_j}<\tfrac12, j=1,\ldots,n$. This allows us to define $T_{\sph}(v_{i,0})$ as a holomorphic function in variables $b_1,\ldots,b_n$ on the set $\abs{\Re b_j}<\tfrac12, j=1,\ldots,n$. We claim that $T_{\sph}$ can be extended to the whole $S'$ by linearity. Indeed, if we have any linear dependence $\sum S_{i}v_{i,0}=0$, $S_i\in \C[b_1,\ldots,b_n]$, then $T_{\sph}(\sum S_i v_{i,0})$ vanishes for generic $\abs{\Re b_j}<\tfrac12$, hence is identically zero.

It follows from BFN construction that $\mc{A}_{G,N}$ is free over $\C[b_1,\ldots,b_n]$. Take any basis $\{u_j\}$ of $\mc{A}_{G,N}$. Since $\mc{A}_{G,N}=S\subset S'$, we can write $u_j=\sum R_{i,j,\lambda}v_{i,\lambda}$ and define $T(u_j)$ as a holomorphic function on the set $\abs{\Re b_k}<\tfrac12, k=1,\ldots,n$.

 It remains to check that $T(v_{i,\lambda}v_{j,-\lambda})=T(v_{j,-\lambda}g_{\text{parity}}(v_{i,\lambda}))$. We checked this in the case when $b_k$ are generic. Since both sides are holomorphic on the set $\abs{\Re b_k}<\tfrac12,k=1,\ldots,n,$ this equality works in general.
\end{proof}
Restricting to cohomology of a point we get the following theorem:
\begin{thr}
\label{ThrSphereTraceCohomology}
Suppose $\mc{M}_{G,N}$ is conical and $\abs{\Re b_i}<\tfrac12$ for $i=1,\ldots,n$. Then there exists a $g_{\text{parity}}$-twisted trace $T_{\sph}\colon \mc{A}_{G,N}\to \C$ such that for $R(x)\in H^*_G(pt)\subset \mc{A}_{G,N}$ we have
\[T(R(x))=\int_{\mf{t}_{\R}}R(x)w_0(x)dx,\] where \[w_0(x)=\frac{\prod_{\alpha>0} \sinh(\pi\alpha)^2}{\prod_{j=1}^n \cosh(\pi(\xi_j-\rmi b_j))}.\]
\end{thr}
\begin{proof}
It is enough to add the basis of $H^*_G(pt)$ to the set $v_{i,0}$ in the proof of Proposition~\ref{PropTSph}.
\end{proof}
\begin{rem}
Let us show that in the abelian case the sphere trace can be analytically continued. Let $G=T=(\C^{\times})^d$ be a torus. Suppose that $\xi_1,\ldots,\xi_d$ are linearly independent. Then we can assume that $b_1=\cdots=b_d=0$. We have
\[T(R(x))=\int_{\mf{t}_{\R}}\frac{R(\rmi x)dx}{\prod_{i=1}^n \cosh(\pi(\xi_i-\rmi b_i))}=(R(\partial_y)\mc{F}_xw_0(x))|_{y_1=\cdots=y_d=0}.\] Here $\mc{F}_x$ means that the Fourier transform is applied only in $x$ varibles and $y_1,\ldots,y_d$ are the variables after the Fourier transform. So we need to understand the behavior of $\mc{F}_xw_0(x)$ with respect to $b_{d+1},\ldots,b_n$. 

Consider the full Fourier transform $\mc{F}w_0(x,b_{d+1},\ldots,b_n)$. Now $\xi_1-\rmi b_1=\xi_1,\ldots,\xi_n-\rmi b_n$ form a linear basis of $\C^n$. Since $\cosh^{-1}(\pi z)$ is self-dual, $\mc{F}w_0=\prod \cosh(\pi\zeta_i)$ (both statements up to a constant), where $\zeta_i$ is a basis dual to $\xi_i-\rmi b_i$. This basis can be computed as follows. Let $\xi_1^*,\ldots,\xi_d^*\in\mf{t}$ be the basis dual to $\xi_1,\ldots,\xi_d$ and the same for $b_{d+1}^*,\ldots,b_n^*$. Then $\zeta_i=\rmi b_i^*$ if $i>d$ and $\zeta_i=\xi_i^*-\rmi\sum_{j>d}b_j^*\xi_i^*(\xi_j)$ if $i\leq d$.

This means that for some positive integer $N$ the function $\mc{F}w_0$ in the variables $\frac{b_{d+1}^*}{N},\ldots,\frac{b_n^*}N$ defines a trace on some abelian Coulomb branch. In the case when $\xi_1,\ldots,\xi_n$ form a basis of weight lattice we can take $N=1$ and get a sphere trace on a dual Coulomb branch. (A similar result should hold in the case when $\xi_1,\ldots,\xi_n$ span weight lattice.) In particular, we can apply Theorem~3.2 in~\cite{KAb} to $\mc{F}_xw_0=\mc{F}_{b^*}\mc{F}w_0|_{b\mapsto -b}$ to get that $\mc{F}_xw_0$ depends meromorphically on $b_{d+1},\ldots,b_n$.
\end{rem}
\appendix
\section{Action on the cohomology of a point}
\label{SecCohomology}
This construction is introduced on pages 36-37 of~\cite{LW}, here we provide more details and show compatibility with the localization.

Consider the diagram

\[
\begin{tikzcd}
\bullet/G(O) & \ar[l,"p_1",shift left]\ar[l,"p_2"', shift right] G(O)\backslash G(K)/G(O),
\end{tikzcd}
\]
where $p_1$ collapses $G(K)/G(O)$ and $p_2$ collapses $G(O)\backslash G(K)$.

We define the action of $K$-theoretic quantized pure Coulomb branch algebra \[\mc{A}_{G,0}^q=K^{\C^{\times}_q}(G(O)\backslash G(K)/G(O))\] on $K^{\C^{\times}_q}(\bullet/G(O))$ by \[\mc{G}.\mc{F}=p_{1,*}(\mc{G}\otimes p_2^*\mc{F}).\] We note that this action is compatible with the $K^{\C^{\times}_q}(\bullet/G(O))$-module structure on $\mc{A}_{G,0}$ given by $\mc{F}.\mc{G}=\mc{G}\otimes p_1^*\mc{F}$ by projection formula.

Let us check that this indeed gives the action of $\mc{A}_{G,0}$ on $K^{\C^{\times}_q\times G}(pt)$. Consider the diagram
\[
\begin{tikzcd}
& G(O)\backslash G(K)\times^{G(O)}G(K)/G(O)\ar[dl,"\pi_2"]\ar[d,"\pi_3"]\ar[dr,"\pi_1"] & \\
G(O)\backslash G(K)/G(O) \ar[d,"p_1"]\ar[dr,"p_2"] & G(O)\backslash G(K)/G(O) \ar[dl,"p_1"]\ar[dr,"p_2"] & G(O)\backslash G(K)/G(O)\ar[dl,"p_1"]\ar[d,"p_2"]\\
\bullet/G(O) & \bullet/G(O) & \bullet/G(O)
\end{tikzcd}
\]
Then we get

\begin{align*}
\mc{H}.(\mc{G}.\mc{F})=p_{1*}(\mc{H}\otimes p_2^*p_{1*}(\mc{G}\otimes p_2^*\mc{F})&)= p_{1*}(\mc{H}\otimes \pi_{2*}\pi_1^*(\mc{G}\otimes p_2^*\mc{F}))=\\
p_{1*}\pi_{2*}(\pi_2^*\mc{H}\otimes \pi_1^*\mc{G}\otimes \pi_1^*p_2^*\mc{F}&)= p_{1*}\pi_{3*}(\pi_2^*\mc{H}\otimes \pi_1^*\mc{G}\otimes \pi_3^*p_2^*\mc{F})=\\
p_{1*}(\pi_{3*}(\pi_2^*\mc{H}\otimes \pi_1^*\mc{G})\otimes p_2^*\mc{F}&)= p_{1*}(\mc{H}*\mc{G}\otimes p_2^*\mc{F})=(\mc{H}*\mc{G}).\mc{F}
\end{align*}

Let us compute the action of $\mc{A}_{T,0}$ explicitly.

\begin{prop}
\label{PropActionTZero}
Consider the $K$-theory class $r^{\lambda}$ of the scyscraper sheaf of $z^{\lambda}$. Then $r^{\lambda}.R(z)=R(q^{\lambda}z)$.
\end{prop}
\begin{proof}
I have learned a similar proof on Gus Schrader's lectures on $K$-theoretic Coulomb branches.

Take $R(z)=[\C_{\xi}]$, where $\xi$ is a character of $T$. Thinking of $K^{\C^{\times}_q}(T(O)\backslash T(K)/T(O))$ as $T(O)\rtimes \C^{\times}_q$-equivariant sheaves on $T(K)/T(O)$, we have $p_2^*(\mc{F})\otimes \mc{G}$ is the representation $\C_{\xi}$ at the point $z^{\lambda}$. 

Pulling further back to $T(K)$ we get the structure sheaf of $T(O)\times T(O)$-invariant subset $z^{\lambda}T(O)$ with the natural equivariant structure twisted by $\C_{\xi}$, a representation of left $T(O)$. This sheaf (before twist) corresponds to the module $M\cong \C[T(O)]$ with the natural left and right $T(O)$-module structure. The dependence on $\lambda$ is in the action of the loop rotation group: $s\in \C^{\times}_q$ acts on the function $t_{i,j}$ (meaning $t_iz^j$ goes to one, all other go to zero, hence $t_{i,0}$ is invertible) by $s^{\lambda_i+j}$.

Now we want the left $T(O)$ action to become the natural one with no twist. We choose another generator of $M$, namely $m_0=\prod t_{i,0}^{-\xi_i}$. Then (notational issues may remain) \[a(z)f(b(z))m_0=\big(f(a(z)b(z))\prod a_{i,0}^{-\xi_i}m_0\big) \prod a_{i,0}^{\xi_i}=f(a(z)b(z))m_0,\] the standard action of $a(z)\in T(O)$ on $f\in \C[T(O)]$.

This changes the loop rotation action by $q^{\xi(\lambda)}$. The action of right $T(O)$ is now twisted by $\C_{\xi}$. Hence after taking quotient by the left $T(O)$ we get a scyscraper sheaf on $z^{\lambda}$ twisted by $\C_{\xi}$. Its pushforward is $\C_{\xi}$. Considering the loop rotation, we get (notation again)
\[r^{\lambda}.s^{\xi}=q^{\xi(\lambda)}s^{\xi},\] as claimed.
\end{proof}
Since $r^{\lambda}$ form a basis of $\mc{A}_{T,0}$ we get an injective homomorphism from $\mc{A}_{T,0}$ to $\End(\C[\mf{t}])$. 
It remains to show that the localization map is compatible with the action on the $K$-theory of a point. Let $i\colon \Gr_T\to \Gr_G$ be the inclusion of affine Grassmannians.

\begin{lem}
\label{LemLocalizationIsCompatibleWithAction}
Consider the map $(i_*)\colon \mc{A}_{T,0}^W\to \mc{A}_{G,0}$. Then for any $a\in \mc{A}_{T,0}^W$ and any $m\in K^G(pt)\subset K^T(pt)$ we have $am=(i_*a)m$.
\end{lem}
\begin{proof}
Consider the diagram
\[
\begin{tikzcd}
T(O)\backslash\bullet\ar[d,"q"] & \ar[l,"p_1^T"] T(O)\backslash T(K)/T(O)\ar[r,"p_2^T"]\ar[d,"i"] & \bullet/T(O)\ar[d]\\
G(O)\backslash\bullet\ar[d] & \ar[l,"p_1^G"] G(O)\backslash G(K)/T(O)  \ar[r,"p_2^G"]\ar[d,"\pi"] &\bullet/T(O)\ar[d,"q"]\\
G(O)\backslash\bullet & \ar[l,"p_1"] G(O)\backslash G(K)/G(O) \ar[r,"p_2"] & \bullet/G(O)
\end{tikzcd}
\]
Note that $q^*$ gives precisely the inclusion $K^G(pt)\subset K^T(pt)$.

First, we check that $i_*$ is compatible with the action on $\mc{F}\in K^{G(O)}(pt)$, namely
\[p_{2*}^T(p_1^{T*}q^*\mc{F}\otimes \mc{G})=p_{2*}^G(p_1^{G*}\mc{F}\otimes i_*\mc{G}).\] First, we have $p_2^T=p_2^Gi$, so it is enough to check
\[i_*(p_1^{T*}q^*\mc{F}\otimes \mc{G})=p_1^{G*}\mc{F}\otimes i_*\mc{G}.\] Then the top left corner commutes, hence 
\[p_1^{T*}q^*\mc{F}=i^*p_1^{G*}\mc{F}.\] The claim now follows from the projection formula.

Now we check that $\pi^*$ is compatible with the action on $\mc{F}\in K^{G(O)}(pt)$, namely
\[p_{2*}^G(p_1^{G*}\mc{F}\otimes\pi^*\mc{G})=q^*p_{2*}(p_1^*\mc{F}\otimes\mc{G}).\] We have $p_1^G=p_1\pi$, hence 
\[p_1^{G*}\mc{F}=\pi^*p_1^*\mc{F},\] and 
\[p_1^{G*}\mc{F}\otimes\pi^*\mc{G}=\pi^*(p_1^*\mc{F}\otimes \mc{G}).\] It remains to use the equation $p_{2*}\pi^*=q^*p_{2*}$. This equation means that leaving only $T(O)$ part of $G(O)$-equivariant structure and then doing pushforward is the same as doing pushforward and then leaving only $T(O)$ part of $G(O)$-equivariant structure.

The lemma follows.

\end{proof}

\end{document}